\documentclass{amsart}

\usepackage{amssymb}

\usepackage[cmtip,all]{xy}

\newtheorem{theorem}{Theorem}[section]

\newtheorem{proposition}[theorem]{Proposition}
\newtheorem{corollary}[theorem]{Corollary}

\theoremstyle{definition}
\newtheorem{definition}[theorem]{Definition}
\newtheorem{example}[theorem]{Example}

\newtheorem{question}[theorem]{Question}
\newtheorem{remark}[theorem]{Remark}

\DeclareMathOperator{\m}{\mathbf m}

\DeclareMathOperator{\n}{\mathbf n}

\DeclareMathOperator{\gr}{gr}

\DeclareMathOperator{\Ass}{Ass }

\DeclareMathOperator{\ord}{ord }

\DeclareMathOperator{\goto}{goto}

\begin{document}

\title[The Goto numbers of parameter ideals]{The Goto numbers of parameter ideals}

\author[William Heinzer]{William Heinzer}
\address{Department of Mathematics,   Purdue University, West Lafayette, IN 47907 }
\curraddr{} \email{heinzer@math.purdue.edu}
\thanks{}

\author[Irena Swanson]{Irena Swanson}
\address{Department of Mathematics, Reed College, Portland, OR 97202 }
\curraddr{} \email{iswanson@reed.edu}
\thanks{}



\dedicatory{}

\begin{abstract} Let $Q$ be a parameter ideal of a Noetherian local
ring $(R,\m)$. The Goto number $g(Q)$ of $Q$ is the  largest integer
$g$ such that $Q:\m^g$ is integral over $Q$. We examine the values
of $g(Q)$ as $Q$ varies over the parameter ideals of $R$. We
concentrate mainly on the case where $\dim R = 1$, and many of our
results concern parameter ideals of a numerical semigroup ring.
\end{abstract}

\maketitle

\baselineskip 17 pt
\section{Introduction}
 This note started from the group work at the
workshop ``Integral closure, multiplier ideals, and cores" that took
place at the American Institute of Mathematics (AIM) in Palo Alto,
California, in December 2006. Shiro Goto presented the background,
motivation, and some intriguing open questions.

Recall that if $(R,\m)$ is a Noetherian local ring with $\dim R =
d$, then an $\m$-primary ideal $Q$ is called a {\it parameter ideal}
if $Q$ is generated by $d$ elements.

A motivating result for the group work at AIM is:

\begin{theorem} {\em
(Corso, Huneke, Vasconcelos \cite{CHV}, Corso, Polini \cite{CP},
Corso, Polini, Vasconcelos \cite{CPV}, Goto \cite{G})} \label{1.1}
Let $(R,\m)$ be a Cohen-Macaulay local ring  of positive dimension.
Let $Q$ be a parameter ideal in $R$ and let $I = Q : \m$. Then the
following are equivalent:
\begin{enumerate}
\item $I^2 \not = Q I$.
\item  The integral closure of $Q$ is $Q$.
\item $R$ is a regular local ring and $\mu(\m/Q) \le 1$.
\end{enumerate}
\end{theorem}

Consequently, if $(R,\m)$ is a Cohen-Macaulay local ring that is not
regular, then $I^2 = QI$. If $\dim R > 1$, it follows that the Rees
algebra $R[It]$ is a Cohen-Macaulay ring,  and even without the
assumption that $\dim R > 1$, the fact that $I^2 = QI$ implies that
the associated graded ring $\gr_I(R) = R[It]/IR[It]$ and the fiber
ring $R[It]/\m R[It]$ are both Cohen-Macaulay.

In \cite{GMT}, Goto, Matsuoka, and Takahashi explore  the
Cohen-Macaulayness and Buchsbaumness of the associated graded and
fiber rings and Rees algebras for ideals $I = Q : \m^2$ under the
condition that $I^3 = QI^2$. They also give  examples showing that
Cohen-Macaulayness does not always hold. Notice that the condition
$I^3 = QI^2$ implies that $I$ is integral over $Q$, so $I \subseteq
\overline Q$, where $\overline Q$ denotes the integral closure of
$Q$ \cite[Corollary 1.1.8]{SHbk}.

It seems that a natural  next step would be to explore the
Cohen-Macaulay property for the various ring constructs from the
ideal $I = Q : \m^3$. We expect the necessity of even further
restrictions on $R$ and $I$. However, rather than examining each of
$I = Q : \m^i$ for increasing $i$ in turn, we pass to examining $I =
Q : \m^g$, where $g$ is the greatest integer such that $Q : \m^g$ is
integral over $Q$. Because of the pioneering  work Shiro Goto has
done in this area we define the  Goto number  of a parameter ideal
$Q$ as follows:

\begin{definition} Let $Q$ be a parameter ideal of the Noetherian
local ring $(R,\m)$. The largest integer $g$ such that $Q:\m^g$ is
integral over $Q$ is denoted $g(Q)$ and called the {\it Goto number
} of $Q$.  In the case where $\dim R = 1$ and $Q = xR$, we sometimes
write $g(x)$ instead of $g(Q)$.
\end{definition}

Notice  that the Goto number $g(Q)$ is well defined, for  $Q : \m^0
= Q : R = Q$ is integral over $Q$, and for sufficiently large $n$,
$\m^n \subseteq Q$, so $Q : \m^n = R$, which is not integral over
$Q$.

During the workshop we concentrated on various invariants, dubbed
``Goto invariants of a Noetherian local ring $(R,\m)$", that involve
the Goto numbers of parameter ideals.  These invariants are
discussed in Section \ref{highdim}. During our subsequent work, we
decided that the set
$$
\mathcal G(R) =  \{g(Q) ~|~ Q \mbox{ is a parameter ideal of } R \},
$$ where $R$ is a
fixed  one-dimensional Noetherian local ring is a possibly more
interesting invariant. Most of the paper has to do with an
examination of the integers that are in $\mathcal G(R)$.  In the
case where $(R,\m)$ is an arbitrary one-dimensional Noetherian local
ring,   we prove  the existence of a positive integer $n$ such that
every parameter ideal contained in $\m^n$ has Goto number the
minimal integer in $\mathcal G(R)$.  With additional hypothesis on
$R$, we prove that the set $\mathcal G(R)$ is finite.

Our notation is mainly as in \cite{SHbk}. In particular, we use
$\overline{R}$ to denote the integral closure of the ring $R$,  and
$\overline{J}$ to denote the integral closure of the ideal $J$ of
$R$. For many of the examples in the paper, the calculations were
done using the symbolic computer algebra system Macaulay2 \cite{GS}.

For much of the paper we focus on a special type of one-dimensional
Noetherian local domain.  As in the monograph of J\"urgen Herzog and
Ernst Kunz \cite{HK}, we consider a rank-one discrete valuation
domain $V$ with field of fractions  $K$ and let $v : K \setminus
\{0\} \to \mathbb Z$ denote the normalized valuation associated to
$V$. Thus if $x \in V$ generates  the maximal ideal of
$V$, then $v(x) = 1$.  Associated with each subring $R$ of $V$ is
a subsemigroup $G(R) = \{ v(r) ~ | ~ r \in R \setminus \{0\}~ \}$
of the additive semigroup
$\mathbb N_0$ of nonnegative integers. $G(R)$ is the {\it value semigroup} of
$R$ with respect to $V$.

\begin{definition}\label{1.3}  A subring $R$ of $V$ is called a {\it numerical semigroup ring
associated to } $V$  if it satisfies the following properties:
\begin{enumerate}
\item $R$ has field of fractions $K$ and the integral closure of $R$ is
$V$.
\item $V$ is a finitely generated $R$-module.
\item There exists $x \in V$ with $v(x) = 1$ such that $x^n \in R$ for
each integer $n$ such that $n = v(r)$ for some $r \in R$, and if
$\m = xV \cap R$,
then the canonical injection $R/\m  \hookrightarrow V/xV$ is
an isomorphism.
\end{enumerate}
The value semigroup $G(R)$ is the {\it numerical semigroup
associated to}  $R$

\end{definition}

\begin{remark} \label{1.4}
Let  $R$ be a numerical semigroup ring associated to the valuation
domain $V$ as in (\ref{1.3}).  We then have the following.
\begin{enumerate}
\item  Since $V$ is a finitely generated $R$-module, $R$
is a one-dimensional Noetherian local domain with maximal ideal $\m$ \cite[Theorem 3.7]{Mat}.
\item  Since the conductor \cite[page 234]{SHbk}  of $R$ in $V$ is nonzero,
the value semigroup   $ G(R)  = \left\{ v(r) ~ | ~ r \in R
\setminus\{0\} \right\} $ contains all sufficiently large integers.
The largest integer $f$ that is not in $G(R)$ is called the {\it
Frobenius number} of $R$, and $C := x^{f+1}V$ is the conductor of
$R$ in $V$.
\item If $0 < a_1 < a_2 < \cdots < a_d$ are elements of $G(R)$ that
generate $G(R)$, then $\m = (x^{a_1}, x^{a_2}, \ldots, x^{a_d})R$.
\item An application of Nakayama's lemma  \cite[Theorem 2.2]{Mat}
implies $R[x] = V$.
\item If $u$ is a unit of $V$, then $R/\m = V/xV$ implies there
exists a unit $u_0$ of $R$ such that $u - u_0 \in xV$. If $u \ne
u_0$, there exists a positive integer $i$ such that $u - u_0 =
wx^i$, where $w$ is a unit of $V$. Repeating the above process on
$w$, we see that every unit $u$ of $V$ has the form
$$
u = u_0 + u_1x + \cdots + u_fx^f + \alpha,
$$
where $\alpha \in C$, $u_0$ is a unit of $R$, and each $u_i, ~ 1 \le
i \le f$,  is either zero or a unit of $R$.

\item
Every nonzero element $r \in R$ has the form $r = u x^b$ for some $b
\in G$ and some unit $u \in V$. Multiplying $u$ by a unit in $R$ and
using item ~ (5),  we see that every nonzero principal ideal of $R$
has the form $ux^bR$, where
$$
u = 1 + u_1x + u_2x^2 + \cdots + u_fx^f + \alpha,
$$
where $\alpha \in C$ and each $u_i$ is either zero or a unit of $R$.
Thus
$$
ux^b = (1 + u_1x + u_2x^2 + \cdots + u_fx^f)x^b + \alpha x^b.
$$
Since $ux^b \in R$, it follows that $b + i \in G$ for each $i$ such
that $u_i \ne 0$. Also $\alpha \in C$ implies $\alpha = u \beta$,
where $\beta \in C$. Thus
$$
ux^b - \alpha x^b = ux^b - u\beta x^b = ux^b(1-\beta).
$$
Since $1-\beta$ is a unit of $R$, we conclude that each nonzero
principal ideal of $R$ has the form $(1 + u_1x + \cdots +
u_fx^f)x^bR$, where $b \in G$, each $u_i$ is either zero or a unit
of $R$, and if $u_i \ne 0$, then $b + i \in G$.

\item
With $r = ux^b$, if we pass to integral closure, we have
$$
\overline{(r)} = \overline{(r) V} \cap R = \overline{(x^b) V} \cap R
= (x^e : e \in G, e \ge b)R.
$$

\end{enumerate}
\end{remark}

\begin{remark} \label{1.5}
With additional assumptions about the rank-one discrete valuation
domain $V$ it is possible to realize numerical semigroup rings by
starting with the group. Let $k$ be a field and let $x$ be an
indeterminate over $k$. If $V$ is either the formal power series
ring $k[[x]]$ or the localization of the polynomial ring $k[x]$ at
the maximal ideal generated by $x$, then for each subsemigroup $G$
of $\mathbb N_0$ that contains all sufficiently large positive
integers, there exists a numerical semigroup ring $R$ associated to
$V$ such that $G(R) = G$. In each case one takes generators $a_1,
\ldots, a_d$ for $G$. If $V$ is the formal power series ring
$k[[x]]$, then $R = k[[x^{a_1}, \ldots, x^{a_d}]]$ is the  subring
of $k[[x]]$ generated by all power series in $x^{a_1}, \ldots,
x^{a_d}$, while if $V$ is  $k[x]$ localized at the maximal ideal
generated by $x$, then $R$ is $k[x^{a_1}, \ldots, x^{a_n}]$
localized at the maximal ideal generated by $x^{a_1}, \ldots,
x^{a_d}$.
\end{remark}

We observe in Proposition \ref{1.6} a useful result for computing
Goto numbers of parameter ideals in dimension one.

\begin{proposition} \label{1.6} Let $Q_1$ and $Q_2$ be ideals of
a Noetherian local ring  $(R,\m)$. Assume that $Q_2$ is not
contained in any minimal prime of $R$. If $e$ is a positive integer
such that $Q_1:\m^e$ is not integral over $Q_1$, then $Q_1Q_2:\m^e$
is not integral over $Q_1Q_2$.
\end{proposition}

\begin{proof}  It suffices to check integral closure modulo each
minimal prime ideal, so we may assume that $R$ is an integral domain
 \cite[Proposition~1.1.5]{SHbk}.
Let $x \in Q_1 : \m^e$. Then $x Q_2 \subseteq Q_1 Q_2 : \m^e$. If
all the elements in $x Q_2$ are integral over $Q_1 Q_2$, then
\cite[Corollary~6.8.7]{SHbk} implies that $x$ is integral over
$Q_1$.
\end{proof}

In dimension one, the product of two parameter ideals is again a
parameter ideal. Thus  Proposition \ref{1.6} has the following
immediate corollary.

\begin{corollary} \label{1.7}  Let $(R,\m)$ be a one-dimensional
Noetherian local ring. If $Q_1$ and $Q_2$ are parameter ideals of
$R$, then $g(Q_1Q_2) \le \min\{g(Q_1), g(Q_2) \}$.
\end{corollary}

A strict inequality may hold in Corollary \ref{1.7} as we illustrate
in Example \ref{1.8}.

\begin{example} \label{1.8} Let $G = \langle 3, 5 \rangle$ be the
numerical subsemigroup of $\mathbb N_0$ generated by $3$ and $5$,
and let $R$ as in Remark \ref{1.5} be a numerical semigroup ring
such that $G(R) = \langle 3, 5 \rangle$. A direct computation shows
that the parameter ideal $Q = x^5R$ has Goto number $g(x^5) = 3$,
while $Q^2 = x^{10}R$ has the property that $x^9 \in x^{10}R:\m^3$.
Therefore $x^{10}R:\m^3$ is not integral over $x^{10}R$ and
$g(x^{10}) = 2$.

\end{example}

The Goto numbers of parameter ideals of a Gorenstein local ring  may
be described using duality as in Proposition~\ref{1.9}.

\begin{proposition} \label{1.9}  Let $Q$ be a parameter ideal of a
Gorenstein local ring $(R,\m)$. Assume that $Q \subsetneq
\overline{Q}$.
Let $J = Q:\overline{Q}$. Then
$$
g(Q) = \max\{i ~| ~ J \subseteq \m^i + ~  Q \}.
$$

\end{proposition}

\begin{proof}
Since $R/Q$ is a zero-dimensional Gorenstein local ring, $(Q:J) = \overline{Q}$,
and $(Q:\m^i) \subseteq \overline{Q}$ if and only if
$J \subseteq \m^i + ~ Q $, cf. \cite[(3.2.12)]{BH}.
\end{proof}

\section{Goto invariants of local rings need not be bounded}\label{highdim}

Since  a regular local ring of dimension one is a rank-one discrete
valuation domain, the Goto number of every parameter ideal  is $0$
in this case. We prove below that in a two-dimensional regular local
ring, the Goto number of a parameter ideal $Q$ is precisely $\ord Q
- 1$, where $\ord Q$ is the highest power of $\m$ that contains $Q$.
Thus in a two-dimensional regular local ring, the Goto number of a
parameter ideal is uniquely determined by the order of the parameter
ideal. It seems natural to expect at least for many local rings
$(R,\m)$ that the Goto number $g(Q)$ becomes larger as $Q$ is in
higher and higher powers of $\m$. The following are several
invariants of a local ring  $(R,\m)$ involving Goto numbers $g(Q)$
of parameter ideals $Q$ of $R$.

\begin{eqnarray*}
 \goto_1(R)    &=&  \sup \left\lbrace  \frac{ g(Q)}{\ord(Q)}~~~~ | ~~~~ Q \hbox{
varies over parameter ideals of $R$}\right\rbrace,      \\
 \goto_2(R)    &=& \sup \left\lbrace \frac{ g(Q)}{\ord(Q : \m)}
 ~~~~| ~~~~
\hbox{ $Q$ varies over parameter ideals of $R$} \right\rbrace, \\
\goto_3(R) &=& \sup \left\lbrace\frac{ g(Q)}{ \ord(Q : \m^{g(Q)})}
~~~~| ~~~~ \hbox{ $Q$ varies over parameter ideals of $R$}
\right\rbrace.
\end{eqnarray*}

In order to avoid division by zero, in the definition of
$\goto_2(R)$, we exclude the case where $R$ is a regular local ring
and $Q = \m$.

Example \ref{2.2} demonstrates the existence,
for every integer $d \ge 3$,
of a regular local ring $(R,\m)$ of dimension  $d$ for
which each of the invariants $\goto_i(R), i \in \{1,2,3\}$, is
infinity.

%
%
%

\begin{example} \label{2.2}
Let $k$ be a field, $d$ an integer $ > 2$,
$x_1, \ldots, x_d$ variables over $k$.
Let $n \ge e$ be  positive integers,
and let  $Q = (x_1^e, x_2^n, \ldots, x_d^n)$.
Then $g(Q) = (d-2)(n-1) + e - 1$.  For we have:
$$
(x_1^e, x_2^n, \ldots, x_d^n) : (x_1, \ldots, x_d)^{(d-2)(n-1)+e-1}
= (x_1^e) + (x_1, \ldots, x_d)^n,
$$
which is integral over $Q$, and  $$
(x_1^e, x_2^n, \ldots, x_d^n) :
(x_1, \ldots, x_d)^{(d-2)(n-1)+e}$$ contains $x_2^{n-1}$, which is
not integral over $Q$. Furthermore,
$$\ord(Q) = \ord(Q : \m) = \ord(Q
: \m^{(d-2)(n-1) + e - 1}) = e.$$
Thus,  for each $i \in \{1,2,3\}$,
$\goto_i(R) \ge \frac{(d-2)(n-1) + e - 1}{e}$ for all $n \ge e$.
Since $d > 2$, we have   $\goto_i(R) = \infty$.
\end{example}

In the case where $(R,\m)$ is a two-dimensional regular local ring,
we prove in Theorem \ref{2.3} that the Goto number of a parameter
ideal $Q$ depends only on the order of $Q$.

\begin{theorem} \label{2.3} Let $(R,\m)$ be a two-dimensional regular local ring. Then for
each  parameter ideal  $Q$ of  $R$, the Goto number $g(Q) = \ord(Q)
- 1$.
\end{theorem}

\begin{proof}
Passing to the faithfully flat extension $R[X]_{\m R[X]}$ preserves
the parameter ideal property and its order and Goto number, so that
without loss of generality we may assume that $R$ has an infinite
residue field. Let $k = \ord Q$. The proof of \cite[Theorem
3.2]{Wang} shows that $k - 1 \le g(Q)$. (In Wang's notation in that
proof, it is shown that $Q : \m^{k- 1} \subseteq (Q\m^{k- 1} : \m^{k-
1}) \subseteq \overline Q$.) Now we prove that $g(Q) \le k-1$. Let $Q = (a,b)$.
Let $x \in \m \setminus \m^2$ be such that $\ord(a) = \ord(a(R/(x)))$
and $\ord(b) = \ord(b(R/(x))$. Since the residue field of $R$ is
infinite, it is possible to find such an element $x$. The condition
needed for $x$ is that its  image in the associated graded ring
$\gr_{\m}(R)$ is not a factor of the images of $a$ and $b$ in
$\gr_{\m}(R)$). Since $R/(x)$ is a one-dimensional regular local
ring, hence a principal ideal domain, by possibly permuting $a$ and
$b$ we may assume that $b \in (a,x)$. By subtracting a multiple of
$a$ from $b$, without loss of generality $b = b_0 x$ for some $b_0
\in R$. Note that $(a,x) = \m^k + (x)$, and $\ord a = \ord Q = k \le
\ord b$. It follows that $b_0 m^k \subseteq b_0 (a,x) \subseteq (a,
x b_0) \subseteq (a,b)$. However, $b_0 \not \in \overline{(a,b)}$:
otherwise for all discrete valuations $v$ centered on $\m$, $v(b_0)
\ge \min\{v(a), v(b)\}$, whence since $v(b) = v(b_0 x) > v(b_0)$,
necessarily $v(b_0) \ge v(a)$ for all such $v$, so that $b_0 \in
\overline{(a)} = (a)$, contradicting the assumption that $(a,b)$ is
a parameter ideal. This proves that $g(Q) < k = \ord Q$.
\end{proof}

If $(R,\m)$ is  a regular local ring, then the powers of $\m$ are
integrally closed. Hence, in this case, if $Q:\m^i$ is integral over
$Q$, then $\ord Q = \ord (Q:\m^i)$. Thus  Theorem \ref{2.3} implies
the following:

\begin{corollary} \label{2.4} If $(R,\m)$ is a two-dimensional
regular local ring, then each of the invariants $\goto_i(R), i \in
\{1,2,3\}$, is one.
\end{corollary}

\begin{remark} \label{2.5} Let $(\widehat R, \widehat{\m})$ denote
the $\m$-adic completion of the Noetherian local ring $(R,\m)$.
Since $R/I \cong \widehat R/I\widehat R$ for each $\m$-primary ideal
$I$ of $R$, the $\m$-primary ideals of $R$ are in one-to-one
inclusion preserving correspondence with the $\widehat{\m}$-primary
ideals of $\widehat R$. Also, if $I$ is an $\m$-primary ideal, then
 $\overline{I}\widehat R$ is the
integral closure of $I\widehat R$ \cite[Lemma 9.1.1]{SHbk}.
Since $R/\m \cong \widehat R/\widehat{\m}$,
and since each parameter ideal of $\widehat R$
has the form $Q\widehat R$, where $Q$ is a parameter ideal of $R$,
the set $\{\ell_R(\overline{Q}/Q) ~|~ Q \mbox{ is a parameter ideal of } R \}$
is identical to the corresponding set for $\widehat R$.
Since $\widehat R$ is flat
over $R$, we also have $(Q:_R\m^i)\widehat R = (Q\widehat
R:_{\widehat R}\widehat{\m^i})$ for each positive integer $i$.
Therefore,  for each parameter ideal $Q$ of $R$, the Goto number $g(Q)
= g(Q\widehat R)$. Hence  the set
$\mathcal G(R)=  \{g(Q) ~ | ~ Q \mbox{ is a parameter ideal of } R
\}$ is identical to the corresponding set $\mathcal G(\widehat R)$
for $\widehat R$.

\end{remark}

\section{One-dimensional Noetherian  local rings} \label{onediml}

Throughout this section, let $(R,\m)$ be a one-dimensional
Noetherian local ring. In subsequent sections we restrict to the
special case where $R$ is a numerical semigroup ring. If $R$ is a
regular local ring, then it is a principal ideal domain, and hence
the Goto number  $g(Q) = 0$ for every parameter ideal $Q$.  Thus to
get more interesting variations on the Goto number of parameter
ideals, we restrict our attention to non-regular one-dimensional
Noetherian local rings.

Corollary \ref{1.7} is useful for examining the Goto number of
parameter ideals. We observe in Theorem \ref{3.1} that the Goto
number of parameter ideals in a sufficiently high power of the
maximal ideal of $R$  are all the same and that this eventually
constant value is the minimal possible Goto number of  a parameter
ideal of $R$.

\begin{theorem} \label{3.1}
Let $(R,\m)$ be a one-dimensional Noetherian local ring.
\begin{enumerate}
\item If $yR$ is a parameter ideal of $R$, then $g(Q) \le g(y)$ for
every parameter ideal $Q$ such that $Q \subseteq yR$.
\item There exists a positive integer $n$ such that all parameter ideals
of $R$ contained in $\m^n$ have the same Goto number. Moreover, this
number is the minimal Goto number of a parameter ideal of~$R$.
\end{enumerate}
\end{theorem}

\begin{proof} If $Q = qR$ is a parameter ideal and $Q \subseteq yR$,
then $q = yz$ for some $z \in R$. If $Q = yR$, then $g(Q) = g(y)$,
while if $Q$ is properly contained in $yR$, then $zR$ is a parameter
ideal, and Corollary~\ref{1.7} implies that $g(Q) \le g(y)$. This
establishes item~(1). For item~(2), let $yR$ be a parameter ideal
such that $g(y)$ is the minimal element of the set
\[
\mathcal G(R) = \{g(Q) ~ | ~ Q \mbox{ is a parameter ideal of } R
\}.
\]
Since $yR$ is a parameter ideal, there exists a positive integer $n$
such that $\m^n \subseteq yR$. By item ~(1), $g(Q) \le g(y)$ for
every parameter ideal $Q \subseteq \m^n$, and by the choice of
$g(y)$, we have $g(Q) = g(y)$ is the minimal Goto number of a
parameter ideal of $R$.
\end{proof}

\begin{remark} \label{3.2} Let  $g = g(Q)$ denote the Goto number of
the parameter ideal $Q$. The chain of ideals $$ Q = Q:\m^0
\subsetneq Q:\m \subsetneq Q:\m^2 \subsetneq \cdots \subsetneq
Q:\m^g \subseteq \overline Q $$ implies that  the length
$\ell_R(\overline{Q}/Q)$ of the $R$-module $\overline Q/Q$  is an
upper bound on $g(Q)$. Thus if $(R,\m)$ is a one-dimensional
Noetherian reduced ring\footnote{If $(R,\m)$ is a Noetherian local
ring that is not equal to its total quotient ring and if
$\overline{R}$  is module-finite over $R$, then $R$ is reduced. For
if $x \in \m$ is a regular element and $y \in R$ is nilpotent, then
$y/x^n \in \overline{R}$, so $y \in x^n\overline{R}$, for each $n
\in \mathbb N$.  But if $\overline{R}$ is module-finite over $R$,
then $\overline{R}$ is Noetherian and $\bigcap_{n=1}^\infty
x^n\overline{R} = (0)$, cf. \cite[Prop. 1.5.2]{SHbk}.}
 such that $\overline{R}$  is a finitely
generated $R$-module, then the length of $\overline{R}/R$ is an
upper bound for $g(Q)$ and therefore the set $\mathcal G(R)$ is
finite. To see this, let $Q = qR$ be a parameter ideal of $R$. Then
$q\overline R$ is an integrally closed ideal of $\overline R$, and
$\overline Q = q\overline R \cap R$, cf.
\cite[Proposition~1.6.1]{SHbk}. Thus we have
$$
\ell_R(\overline R/R) = \ell_R(q\overline R/qR) \ge \ell_R(\overline
Q/Q).
$$
\end{remark}

\begin{remark}\label{3.3}
For certain parameter ideals $Q$ it is possible to compute the Goto
number $g(Q)$ as an index of nilpotency. If $Q = xR$ is a reduction
of $\m$, then $\m$ is the integral closure of $Q$ and
$$
g(Q) = ~ \max\{ i ~ | ~ (Q :\m^i)  \ne R  \} = ~ \min\{ i ~ | ~
\m^{i+1}  \subseteq Q  \}
$$
is the index of nilpotency of $\m$ with respect to $Q$
\cite[(4.4)]{HKU}. This is an integer that is less than or equal to
the reduction number of $\m$ with respect to $Q$,  with equality
holding if the associated graded ring $\gr_{\m}(R)$ is
Cohen-Macaulay.

\end{remark}

We prove in Theorem \ref{3.4} a sharpening of Theorem \ref{3.1} in
the case where  $\overline{R}$ is module-finite over $R$.

\begin{theorem} \label{3.4}  Let $(R,\m)$ be a one-dimensional Noetherian local
reduced ring  such that $\overline R$ is module-finite over $R$. Let
$C = R :_R \overline R$ be the conductor of $R$ in $\overline{R}$,
and let $x \in \m $ and $y \in C $ generate parameter ideals.  Then
for each positive integer $n$, the Goto number $g(x^ny) = g(xy)$.
Thus for all parameter ideals $Q = qR \subseteq x C =
\overline{xC}$, we have $ g(Q) = g(xy)$. Furthermore, this is the
minimal possible Goto number of a parameter ideal in $R$.
\end{theorem}

\begin{proof}
By Corollary \ref{1.7}, $g(xy) \ge g(x^ny)$. To prove that $g(xy)
\le g(x^ny)$, it suffices to prove for each positive integer $i$
that
\begin{eqnarray} \label{3.11}
(xyR:\m^i) \subseteq xy\overline{R} \implies (x^nyR:\m^i) \subseteq
x^ny\overline{R}.
\end{eqnarray}
Assume there exists $w \in R$ with $w\m^i \subseteq x^nyR$ and with
$w \notin x^ny\overline{R}$. Notice that $xw \in x^ny\overline{R}
\subseteq x^nC \subseteq x^nR$ implies $w \in x^{n-1}R$. Therefore
by replacing $w$ if necessary by $wx^j$ for some positive integer
$j$, we may assume that $w \in x^{n-1}R$, so $w = x^{n-1}z$ for some
$z \in R$. Thus $x^{n-1}z\m^i \subseteq x^nyR$ implies that $z\m^i
\subseteq xyR$, so $z \in xyR:\m^i$.  Moreover, $w = x^{n-1}z \notin
x^ny\overline{R}$ implies that $z \notin xy\overline{R}$. This
establishes the implication displayed in (\ref{3.11}).
Theorem~\ref{3.1} implies that for $n$ sufficiently large $g(x^ny)$
is the minimal Goto number of a parameter ideal of $R$.  This
completes the proof of Theorem~\ref{3.4}.
\end{proof}

In comparison with Theorem \ref{3.4}, we demonstrate  in Example
\ref{4.6} that the Goto number $g(Q)$ of parameter ideals contained
in the conductor need not be constant,  even in the case where
$(R,\m)$ is a Gorenstein numerical semigroup ring.

Theorem \ref{3.5} establishes conditions on a one-dimensional
Noetherian local ring $R$ in order that the set
$\{ \ell_R(\overline Q/Q) ~|~ Q \mbox{ is a parameter ideal of }
R \}$ is finite.

\begin{theorem} \label{3.5}
Let $(R,\m)$ be a one-dimensional Noetherian local ring, let
$(\widehat R, \widehat{\m})$ denote the $\m$-adic completion of $R$,
and let $\n$ denote the nilradical of $\widehat R$.  The
following statements are equivalent.
\begin{enumerate}
\item  The length $\ell_{\widehat R}(\n)$ is finite.
\item The set $
\{ \ell_R(\overline Q/Q) ~ | ~ Q \mbox{ is a parameter ideal of } R
\} $ is finite.
\end{enumerate}
\end{theorem}

\begin{proof}  By
Remark~\ref{2.5},  item (2) holds for $R$ if and only if it holds
for $\widehat R$.  Therefore,  to prove $(1) \iff (2)$, we may
assume that $R$ is complete.

Assume that $\ell_R(\n)$  is finite, and let $R' = R/\n$.  If $Q$ is
a parameter ideal of $R$, then $\n \subset \overline Q$ and
$\ell_R((Q+\n)/Q) \le \ell_R(\n)$. Since $R'$ is a reduced complete
Noetherian local ring, its integral closure is a finite $R'$-module.
Thus by Remark~\ref{3.2}, the set $ \{ \ell_R'(\overline QR'/QR') ~
| ~ Q \mbox{ is a parameter ideal of } R' \} $ is  bounded by some
integer $s$. It follows that $s + \ell_R(\n)$ is an upper bound for
$\ell_R(\overline Q/Q)$, so the set $ \{ \ell_R(\overline Q/Q) ~ | ~
Q \mbox{ is a parameter ideal of } R \} $ is finite.

Assume that $\ell_R(\n)$  is infinite and let $Q_1 = xR$ be a
parameter ideal of $R$. For each positive integer $n$, let
$Q_n = x^nR$.  Then $Q_n + \n \subseteq \overline{Q_n}$,  and
$$
\frac{(Q_n+ \n)}{Q_n} \cong \frac{\n}{(Q_n \cap \n)} =  \frac{\n}{Q_n\n}.
$$
Hence $\ell_R(\overline{Q_n}/Q_n) \ge \ell_R(\n/x^n\n)$. Therefore
$\ell_R(\overline{Q_n}/Q_n)$ goes to infinity as $n$  goes to
infinity. This completes the proof of Theorem~\ref{3.5}.
\end{proof}

\begin{corollary} \label{3.50} With notation as in Theorem~\ref{3.5},
if the length $\ell_{\widehat R}(\n)$ is finite, then the set
$\mathcal G(R) = \{g(Q) ~ | ~ Q \mbox{ is a parameter ideal of }
R \}$ is finite.
\end{corollary}
\begin{proof} Apply Theorem~\ref{3.5} and Remark~\ref{3.2}.
\end{proof}

\begin{remark} \label{3.51}
Let $(R,\m)$ be a one-dimensional reduced Cohen-Macaulay local ring,
and let $\widehat R$ denote the $\m$-adic completion of $R$. If the
nilradical $\n$ of $\widehat R$ is nonzero, then $\ell_{\widehat R}(\n)$
is infinite.  For if $xR$ is a parameter ideal of  $R$, then $x$ is a
regular element of $\widehat R$, and hence $\{x^n\n \}_{n=1}^\infty$ is a
strictly descending chain of ideals of $\widehat R$. It is known that
$\n = (0)$ if and only if $\overline R$ is module finite over $R$.
There are well-known examples of one-dimensional Noetherian local domains
$R$ for which $\overline R$ is not module finite over $R$. For
such a ring $R$, Theorem~\ref{3.5}
implies that the set  $\{ \ell_R(\overline Q/Q) ~|~ Q \mbox{ is a parameter ideal of }
R \}$ is not finite.

A specific example of a one-dimensional Noetherian local domain $R$
for which $\overline R$ is not module finite over $R$ is given by
Nagata \cite[(E3.2), page~206]{N} and described in  \cite[Ex.~4.8,
page~89]{SHbk}. Let $A = k^p[[X]][k]$, where $k$ is a field of
characteristic $p > 0$ such that $[k:k^p] = \infty$. Then $A$ is a
one-dimensional regular local ring. The example of Nagata is
$$
R = \frac{A[Y]}{(Y^p - \sum_{i \ge 1} b_i^p X^{ip})},
$$
where $\{b_i\}_{i=1}^\infty$ are elements of $k$ that are
$p$-independent over $k^p$.

We prove that the set $ \mathcal G(R)$ of Goto numbers of parameter
ideals of $R$ is infinite. By Remark~\ref{2.5}, it suffices to prove
that the completion $\widehat R$ of $R$ has this property.  Notice
that $\widehat R$ is a homomorphic image of a two-dimensional
regular local domain: indeed, with $S = k[[X,Y]]$, then $\widehat R
\cong  S/Y^pS$, so $\widehat R = k[[x,y]]$, where $y^p = 0$.
Corollary~\ref{3.71} below implies that $\mathcal G(R)$ is infinite.
\end{remark}

\begin{theorem} \label{3.7}  Let $(R,\m)$ be a one-dimensional Noetherian
local ring. If there exists  a nonzero principal ideal $yR$ such
that $R/yR$ is one-dimensional and Cohen-Macaulay
and $(0) : y$ is contained in the nilradical, then the set
$\mathcal G(R)$ is infinite.
\end{theorem}

\begin{proof} The assumption that $R/yR$ is one-dimensional and Cohen-Macaulay
implies that each $P \in \Ass R/yR$ is a minimal prime of $R$. Let
$$ x \in \m \setminus \bigcup_{P \in \Ass R/yR} P.
$$
If $R$ has minimal primes other than those in $\Ass R/yR$, choose
$x$ also to be in each of these other minimal primes of $R$.  For
each positive integer $n$, let $Q_n := (y + x^n)R$. Notice that
$Q_n$ is a parameter ideal of $R$. Checking integral closure modulo
minimal primes, we see that $(y, x^n)R + \n \subseteq
\overline{Q_n}$, where $\n$ is the nilradical of $R$. We prove that
$g(Q_n) \ge n$. Let $r \in (Q_n:\m^n)$. Then $r \in (Q_n:x^n)$, so
$rx^n = a(y+x^n)$, for some $a \in R$. Hence $(r-a)x^n = ay$, so
$r-a \in (yR:x^n)$.  Since $x^n$ is regular on $R/yR$, we have  $r-a
= by$, for some $b \in R$. It follows that $x^nby = ay$, so
$(x^nb-a)y = 0$ and $x^nb-a \in (0):y \subseteq \n$.
Therefore $a = x^nb + c$, where $c \in \n$. Hence $r = bx^n + by + c
\in \overline{Q_n}$. We conclude that $g(Q_n) \ge n$, and therefore
that $\mathcal G(R)$ is infinite.
\end{proof}

\begin{corollary} \label{3.71}
Let $(R,\m)$ be a one-dimensional Cohen-Macaulay local ring such that $\m$
is minimally $2$-generated.
The following are equivalent:
\begin{enumerate}
\item  
$\mathcal G(R)$ is finite.
\item 
The $\m$-adic completion $\widehat R$ of $R$ is reduced.
\item 
$\overline{R}$ is module-finite over $R$.
\end{enumerate}
\end{corollary}

\begin{proof}
Assume (1).
By Remark~\ref{2.5},
$\mathcal G(\widehat R)$ is finite.
The structure theorem for complete local rings \cite[(31.1)]{N} implies
that $\widehat R$ is a homomorphic image of a complete regular local
ring.  Since
$\m$ is minimally $2$-generated, we obtain  $\widehat R = S/I$, where $S$ is a $2$-dimensional
regular local ring. Since $\widehat R$ is Cohen-Macaulay and $\dim \widehat R = 1$,
the ideal $I$ is of the form
$I = (p_1^{e_1} \cdots p_t^{e_s})$,
where $p_1, \ldots, p_s$ are non-associate prime elements
and $e_1, \ldots, e_s$ are positive integers.
If $e_i > 1$ for some $i$,
then Theorem~\ref{3.7} applied to $y = p_i$ shows that $\mathcal G(\widehat R)$
is infinite,
which is a contradiction.
So necessarily all $e_i$ equal $1$,
which proves (2).
The implication (2) $\Longrightarrow$ (3)
follows say from \cite[Corollary~4.6.2]{SHbk},
and (3) $\Longrightarrow$ (1)
follows from Remark~\ref{3.2} and Corollary~\ref{3.50}.
\end{proof}

\begin{example} Let $S$ be a $3$-dimensional regular local ring with
maximal ideal $(u,v,w)S$. Let $I = (u,w)S \cap (v^2, u-w)S$ and let
$R = S/I$. Notice that $vR$ is
a nonzero principal ideal such that $R/vR$ is one-dimensional and
Cohen-Macaulay and such that $(0) :_R v$ is contained in the nilradical.
By Theorem~\ref{3.7}, $\mathcal G(R)$ is infinite.
\end{example}

We record in Proposition~\ref{anotherinfinite} a general 
ideal-theoretic condition that implies $\mathcal G(R)$ is infinite.

\begin{proposition} \label{anotherinfinite}
Let $(R,\m)$ be a one-dimensional Noetherian local ring,
and let $x, y$ be elements of $R$
such that for all $n$,
$y + x^n$ is a parameter.
Assume that for all $n$,
$(y) : x^n \subseteq \overline{(y + x^n)}$
and $(x^n) : y \subseteq \overline{(y + x^n)}$.
Then $\mathcal G(R)$ is infinite.
\end{proposition}

\begin{proof}
We prove that $g(y + x^n) \ge n$.
Let $r \in (y+x^n) : \m^n$.
Then $r x^n = a (y + x^n)$ for some $a \in R$.
Then $r - a \in ((y) : x^n) \subseteq \overline{(y + x^n)}$
and $a \in ((x^n): y) \subseteq \overline{(y + x^n)}$,
so that $r \in \overline{(y + x^n)}$.
\end{proof}

We have demonstrated in Remark~\ref{3.51} the existence  of
one-dimensional Noetherian local domains $(R,\m)$ for which the set
$\mathcal G(R)$ of Goto numbers of parameter ideals is infinite.  A
question here that remains open is:

\begin{question} \label{3.52} Let $(R,\m)$ be a one-dimensional
Noetherian local ring. If the set $\mathcal G(R)$ is finite, does
$R$ satisfy the equivalent conditions of Theorem~\ref{3.5} ?
\end{question}

Theorem~\ref{3.7} implies an affirmative  answer to
Question~\ref{3.52} if $R$ is Cohen-Macaulay and $\m$ is
$2$-generated.

In  Proposition \ref{3.8} we obtain  an upper bound on the Goto
numbers  of parameter ideals contained in the conductor in the case
where $R$ is Gorenstein. We thank YiHuang Shen for helpful comments
regarding  Proposition~\ref{3.8}.

\begin{proposition} \label{3.8} Let $(R,\m)$ be a one-dimensional Gorenstein
local reduced ring such that $\overline R$ is module-finite over
$R$. Let $C = R :_R \overline R$ be the conductor of $R$ in
$\overline{R}$, and let $Q = qR$ be a parameter ideal contained in
$C$. Then
$$
g(Q) = \max\{i ~| ~ C \subseteq \m^i + ~  Q \}.
$$
\end{proposition}

\begin{proof}
Since $q \in C$, we have $\overline Q = q\overline R \subseteq C$.
Also, $qC = qC\overline R$, so $QC = \overline Q C$.
Hence $C \subseteq (Q:\overline Q$).
Let $r \in (Q : \overline Q)$.
Then $\frac{r}{q} \overline Q \subseteq R$.
Let $w \in \overline R$.
Then $qw \in C \overline R \cap Q \overline R
\subseteq R \cap Q \overline R = \overline Q$,
whence
$r w = \frac{r}{q} q w \in \frac{r}{q} \overline Q \subseteq R$,
so that $r \in C$.
This proves that $(Q : \overline Q) \subseteq C$
and hence $(Q : \overline Q) = C$.
Now the proposition follows from Proposition~\ref{1.9}.
\end{proof}

\begin{remark} \label{3.9}
The conclusion of Proposition \ref{3.8} fails if $R$ is not assumed
to be Gorenstein. Let $R$ be a numerical semigroup ring associated
to the numerical semigroup generated by $4, 5, 11$. The conductor $C
= R :_R \overline R = x^8\overline R$, and $Q = x^{12}R$ is a
parameter ideal contained in $C$. The Goto number $g(x^{12}) = 2$,
but we have $\max\{i ~ | ~ C \subseteq \m^i + ~ x^{12}R\} = 1$,
because $x^{11} \notin \m^2 + ~x^{12}R$.
\end{remark}

\section{Numerical semigroup rings} \label{sectQ}

This section provides lower and upper bounds on the Goto numbers
of parameter ideals in numerical semigroup rings.

Let $(R,\m)$ be a numerical semigroup ring associated to a rank-one
discrete valuation ring $V$ as in (\ref{1.3}) and let $G$ be the
numerical semigroup associated to $R$.  Assume that $R \subsetneq
V$, or, equivalently, that $G$ is minimally generated by positive
integers $a_1, \ldots, a_d$, with $1 < a_1 < \cdots < a_d$ and
$\gcd(a_1, \ldots, a_d) = 1$. Necessarily $d > 1$, and $\m =
(x^{a_1}, \ldots, x^{a_d})R$ is minimally generated by $x^{a_1},
\ldots, x^{a_d}$.

\begin{theorem} \label{4.1} Let $f$ denote the Frobenius number of
the numerical semigroup ring~$R$. Then
$$
g(x^{f+a_1+1}) = \min\{ g(Q) ~|~ Q \mbox{ is a parameter ideal of }
R \}.
$$
Moreover,  for all $e \ge f + a_1 + 1$, we have $ g(x^e) =
g(x^{f+a_1+1})$.
\end{theorem}

\begin{proof} The conductor $C$ of $R$ into $\overline R = V$
is $C = x^{f+1}V$.   Apply Theorem \ref{3.4}.
\end{proof}

The lower  bound for $e$ given in Theorem \ref{4.1} is sharp: if $G
= \langle 9, 19 \rangle$,  then $f = 143$, $a_1 = 9$, $f + a_1 + 1 =
153$, and $g(x^{152}) = 9 > \min\{g(x^{a_i}) : i = 1, \ldots, d\}
=~8$.

\begin{remark} \label{4.2}
Corollary \ref{1.7} implies  that,  for all $e \ge f + a_1 + 1$, one
has
\begin{eqnarray} \label{4.11}
g(x^e) \le \min\{g(x^{a_i}) ~|~ i = 1, \ldots, d\}.
\end{eqnarray}
We prove  equality holds in (\ref{4.11}) in the case where  $d =
2$ in Theorem~\ref{5.10} below. However, YiHuang Shen has pointed
out that strict inequality may hold in display~(\ref{4.11}) if $d
\ge 3$.  In particular, for the semigroup $\langle 7, 11, 20
\rangle$, one has $g(x^7) = 4$, $g(x^{11}) = 4$ and  $g(x^{20}) =
5$, while  $g(x^{45}) = 3$. Similar strict inequalities occur for
the semigroups $\langle 8, 11, 15 \rangle$, $\langle 9, 14, 17
\rangle$, $\langle 10, 13, 18 \rangle$.  Even in the case where $d =
3$ and the numerical semigroup is symmetric, YiHuang Shen has found
examples where strict inequality holds in display~(\ref{4.11}).
For the symmetric numerical semigroup $\langle 11,14,21 \rangle$,
one has $g(x^{11}) = 6$, $g(x^{14}) = 6$ and $g(x^{21}) = 7$, while
$g(x^{85}) = 5$.
\end{remark}

\begin{proposition} \label{4.31}
Let $(R,\m)$ be a numerical semigroup ring associated to a rank-one
discrete valuation ring $V$ as in (\ref{1.3}) and let $G$ be the
value semigroup of $R$. Then
\begin{eqnarray} \label{4.32}
\sup\{g(x^e) ~|~ e \in G\} ~~=~~
\max\{g(x^{a_j}) ~|~ j = 1, \ldots, d\}.
\end{eqnarray}
\end{proposition}

\begin{proof} Apply Corollary \ref{1.7}.
\end{proof}

We clearly have
\begin{eqnarray} \label{4.33}
\sup\{g(x^e) : e \in G\} \le \sup\{g(Q) ~|~ \hbox{ $Q$ a parameter
ideal in $R$}\}.
\end{eqnarray}
Strict inequality may hold in display (\ref{4.33}) as we demonstrate  in
Example \ref{4.4}.

\begin{example} \label{4.4}  Let $(R,\m)$ be a numerical semigroup ring associated
to the semigroup $G = \langle 4,7,9 \rangle$.  Then $(x^4) : \m^3$
contains $1$,~ $(x^7) : \m^3$ contains $x^4$, and  $(x^9) : \m^3$
contains $x^8$. Therefore display (\ref{4.32}) implies that
$\sup\{g(x^e) : e \in G\} \le 2$. However, $(x^7+x^8+x^9) : \m^3 =
(x^7+x^8+x^9, x^7+x^9+x^{11},x^7-x^{13},x^7-x^{16}, x^9-x^{18})$, so
that $g(x^7+x^8+x^9) \ge 3$.
\end{example}

We  are interested in bounds for the Goto numbers of arbitrary
parameter ideals of a numerical semigroup ring.  Theorem~\ref{4.1}
gives a  general lower bound. Proposition~\ref{4.5} gives a relative
lower bound for each parameter ideal in terms of the Goto number of
the monomial parameter ideal with the same integral closure.

\begin{proposition} \label{4.5}
Let $(R,\m)$ be a numerical semigroup ring associated to a rank-one
discrete valuation ring $V$ as in (\ref{1.3}) and let $G$ be the
value semigroup of $R$. Let $Q = qR$ be a parameter ideal of $R$.
Then $q = ux^b$, where $b \in G$ and $u$ is a unit of $V$, and we
have $g(Q) \ge g(x^b)$.
\end{proposition}

\begin{proof}
Let $r = wx^c \in R$, where $w$ is a unit of $V$ and $c \in G$ with
$c < b$.
It suffices to prove for each positive integer $i$ that $wx^c\m^i
\subseteq ux^bR$ implies that $x^c\m^i \subseteq x^bR$.
Now $\m^i$ is
generated by elements of the form $x^a$, where $a \in G$.
Using part
(5) of Remark \ref{1.4}, we see that $wx^cx^a \in ux^bR$ implies that $c+a-b \in G$, and
this implies that  $x^{c+a} \in x^bR$.
\end{proof}

With notation as in Proposition \ref{4.5}, it may happen that $g(Q)
> g(x^b)$ even in the case where $R$ is Gorenstein and $b > f$, as
we demonstrate in Example \ref{4.6}.

\begin{example} \label{4.6}
Let $(R,\m)$ be a numerical semigroup ring associated to the
semigroup $G = \langle 5,11 \rangle$.  Then $f = 39$, and $g(x^{40})
= 4 < g(x^{40} + x^{44}) = 5$. Note that $x^b = x^{40}$ and $ux^b =
x^{40} + x^{44}$ are in the conductor $C$ of $R$ in $V$.
\end{example}

Theorem \ref{4.7}  is due to Lance Bryant. It gives an upper bound
on the Goto number of parameter ideals.

\begin{theorem} \label{4.7}
Let $(R,\m)$ be a numerical semigroup ring associated to a rank-one
discrete valuation ring $V$ as in (\ref{1.3}). Assume that  $G$ is
minimally generated by  $a_1, \ldots, a_d$, with $1 < a_1 < \cdots <
a_d$, and let $f$ be the Frobenius number of $R$. For all parameter
ideals $Q$ in $R$, we have
\begin{eqnarray} \label{4.77}
g(Q) ~~ \le  ~~ \left\lfloor\frac{f}{a_1}\right\rfloor + 1~~ = ~~\left\lceil\frac{f}{a_1}\right\rceil.
\end{eqnarray}
\end{theorem}

\begin{proof}
Let $Q$ be a parameter ideal of $R$.  As observed in
Remark~\ref{1.4}, $Q = ux^cR$, where $c \in G$, and  $u = 1 +
\sum_{i=1}^f u_i x^i$, where each of the $u_i$ is either zero or a
unit  of $R$. Let $m  = \left\lfloor\frac{f}{a_1}\right\rfloor + 1$.
It suffices to prove that $Q : \m^{m+1}$ contains an element that is
not integral over $Q$. Since $Q$ is a parameter ideal, $c > 0$. Let
$b$ be the largest element in $G$ that is strictly smaller than $c$.
Then $c - b \le a_1$, so $b - c \ge -a_1$. Let $e_i$ be positive
integers such that $\sum e_i = m + 1$. Then
$$
b + \sum e_i a_i - c ~~\ge ~~b + \sum e_i a_1 - c ~~\ge ~~(m) a_1 ~~
> ~~  f. $$ Therefore  $x^b \m^{m+1} \subseteq (x^c)C$,
where $C = R_{> f}$ is the conductor of $R$ in $V$.  Since $x^cC =
(ux^c) u^{-1}C$ and $u^{-1}C \subseteq R$, we have $x^b\m^{m+1}
\subseteq Q$, so $x^b \in Q:\m^{m+1}$. Since $b < c$, the element
$x^b$ is not integral over $Q = ux^cR$. This completes  the proof.
\end{proof}

\begin{remark} \label{4.8}
Let $(R,\m)$ be a numerical semigroup ring associated to the
semigroup $G = \langle a_1, a_2 \rangle$. In Theorem~\ref{5.5}
below, we prove that the Goto number $g(x^{a_2}) =
\left\lfloor\frac{a_2 - b_2 + f }{ a_1}\right\rfloor = a_2 - 1 -
\left\lfloor\frac{a_2 - 1 }{ a_1}\right\rfloor$ is a  sharp upper
bound for the Goto number of  monomial parameter ideals of $R$.
Theorem~\ref{4.7} implies that
$$
\sup\{ g(Q) ~|~ Q \mbox{ is a parameter ideal of } R \} \le
\left\lfloor\frac{f}{a_1}\right\rfloor + 1.
$$
It is well known  that if $G = \langle a_1, a_2 \rangle$, then the
Frobenius number $f = a_1a_2 - a_1 - a_2$, cf.
\cite[Example~12.2.1]{SHbk}. Thus $
\left\lfloor\frac{f}{a_1}\right\rfloor + 1  =
\left\lfloor\frac{a_1a_2-a_1-a_2}{a_1}\right\rfloor + 1$. Since
$\left\lfloor\frac{-a_2}{a_1}\right\rfloor = -\lceil\frac{a_2
}{a_1}\rceil$ and $a_1$ and $a_2$ are relatively prime,   we see
that $g(x^{a_2}) = \left\lfloor\frac{f}{a_1}\right\rfloor + 1$.
Therefore, in the case where $d = 2$, the upper bound given in
Theorem~\ref{4.7} is a sharp upper bound for the Goto numbers of
parameter ideals of $R$,   and this upper bound is attained by the
monomial parameter ideal $(x^{a_2})$.
\end{remark}

\begin{remark} \label{4.9} The upper bound
$g(Q) \le   \left\lceil\frac{f}{a_1}\right\rceil$ given in
Theorem~\ref{4.7} is not always a sharp upper bound for the Goto
numbers of parameter ideals of  a numerical semigroup ring. YiHuang
Shen has constructed a family of examples that illustrate this, the
simplest example being $G = \langle 4, 6, 7 \rangle$. If $(R,\m)$ is
a numerical semigroup ring associated to the semigroup $G = \langle
4, 6, 7 \rangle$, then  $f = 9$, so $3$ is the upper bound given by
Theorem~\ref{4.7}, while $g(Q) = 2$ for each parameter ideal $Q$
of $R$.
\end{remark}

\section{Numerical semigroup rings -- monomial ideals}
\label{numsemi}

As in Section \ref{sectQ}, let $(R,\m)$ be a numerical semigroup
ring associated to a rank-one discrete valuation ring $V$, and let
$G = \langle a_1, \cdots, a_d \rangle $ be the numerical semigroup
associated to $R$. In this section we establish  bounds for the Goto
numbers of monomial parameter ideals in $R$.  It is well known that
numerical semigroups follow varied patterns that are difficult to
classify precisely. For example, in the case where $d \ge 4$, there
is no known closed formula for the Frobenius number of $R$ in terms
of the minimal generators $a_1, \ldots, a_d$ of $G$.

\begin{proposition} \label{5.1}
Let $(R,\m)$ be a numerical semigroup ring associated to the
semigroup $G = \langle a_1,\cdots, a_d \rangle$, and let $f$ denote
the Frobenius number of $R$. For each $j > 1$, we have
\begin{eqnarray} \label{5.11}
g(x^{a_j}) \le \left\lfloor\frac{a_j - b_j + f }{a_1}\right\rfloor,
\end{eqnarray}
where $b_j$ is the largest element of $G$ that is strictly smaller
than $a_j$.
\end{proposition}

\begin{proof}
Set $b = \lfloor\frac{a_j - b_j + f}{a_1}\rfloor$. We prove  that
$(x^{a_j}) : \m^{b+1}$ contains $x^{b_j}$. Let  $c_i \in \mathbb N$
be such that  $\sum_{i=1}^d c_i = b+1$. Then
$$
{b_j} + \sum_{i=1}^d a_i c_i \ge b_j + \sum_{i=1}^d a_1 c_i = {b_j}
+ a_1 (b+1) > b_j + (a_j - b_j + f) = a_j + f.
$$
Since this inequality is strict, $b_j + \sum_{i=1}^d a_i c_i - a_j
\in G$. Therefore
$$x^{b_j} (x^{a_1})^{c_1} \cdots (x^{a_d})^{c_d} \in
x^{a_j}R.
$$
This proves that $(x^{a_j}) : \m^{b+1}$ contains $x^{b_j}$. Since
$b_j < a_j$, the element $x^{b_j}$ is  not integral over $(x^{a_j})$. Thus
$g(x^{a_j}) \le b$.
\end{proof}

The inequality  given in display (\ref{5.11}) may be strict as we
demonstrate in Example~\ref{5.2}.

\begin{example} \label{5.2}
Let $(R,\m)$ be a numerical semigroup ring associated to the
semigroup $G = \langle 9,19,21 \rangle$.  One can compute that the
Frobenius number $f$ of $R$ is $71$. For $j = 3$,
$\left\lfloor\frac{a_j - b_j + f }{ a_1}\right\rfloor =
\left\lfloor\frac{21 - 19 + 71 }{ 9}\right\rfloor = 8$, but
$g(x^{21}) = 6$. However, for $j = 2$, $\left\lfloor\frac{a_j - b_j
+ f}{ a_1}\right\rfloor = \left\lfloor\frac{19 - 18 + 71}{
9}\right\rfloor = 8$ is indeed $g(x^{19})$.
\end{example}

\begin{proposition} \label{5.3}
Let $(R,\m)$ be a numerical semigroup ring associated to the
semigroup $G = \langle a_1,\cdots, a_d \rangle$, and let $f$ denote
the Frobenius number of $R$. Then
\begin{eqnarray} \label{5.33}
g(x^{a_1}) \le \left\lceil\frac{f + a_1 + 1 }{a_2}\right\rceil - 1.
\end{eqnarray}
\end{proposition}

\begin{proof}
It suffices to prove that $1 \in (x^{a_1}) : \m^{\lceil\frac{f + a_1
+ 1}{a_2}\rceil}$, and for this it suffices to prove that whenever
$c_i \in \mathbb N$ and $\sum_i c_i = \lceil\frac{f + a_1 + 1}{
a_2}\rceil$, then $\sum_i c_i a_i - a_1 \in G$. In proving this, we
may assume that  $c_1 = 0$. Then
$$\sum_i c_i a_i - a_1 ~~ \ge ~~ \sum_i
c_i a_2 - a_1~~ \ge ~~ f + a_1 + 1 - a_1~~ > ~~ f.
$$
This completes the proof of Proposition \ref{5.3}.
\end{proof}

In Example \ref{5.2}, where $G = \langle 9,19,21 \rangle$, the
inequality in display (\ref{5.33}) is an equality since  $\lceil\frac{f +
a_1 + 1 }{ a_2}\rceil - 1 = 4 = g(x^9)$. However, if $G = \langle
5, 6, 13 \rangle$ is the value semigroup of $R$, then $f = 14$ and
$\lceil\frac{f + a_1 + 1 }{ a_2}\rceil - 1 = 3
> 2 = g(x^5)$.

Concerning upper bounds for the Goto number of monomial parameter
ideals, as observed in Proposition \ref{4.31}, we have
\begin{eqnarray*}
\rho ~~:=~~\sup \{g(x^e) ~|~ e \in G\}~~ = ~~\max\left\lbrace g(x^{a_1}), \ldots,
g(x^{a_d})\right\rbrace,
\end{eqnarray*}
and Propositions \ref{5.1} and \ref{5.3} imply that
\begin{eqnarray} \label{5.35}
\rho  \le \max\left\lbrace \left\lceil\frac{f
+ a_1 + 1 }{a_2}\right\rceil - 1, \left\lfloor\frac{a_2 - b_2 + f
}{a_1}\right\rfloor, \ldots, \left\lfloor\frac{a_d - b_d +
f}{a_1}\right\rfloor \right\rbrace,
\end{eqnarray}
where $b_j$ is the largest element of $G$ that is strictly smaller
than $a_j$, for each $j$ with $2 \le j \le d$. The  maximum in display (\ref{5.35}) is
at most $1 + \frac{f }{ a_1}$, because
$$
b_i \in \{a_i - a_1, a_i - a_1 + 1, \ldots, a_i-1\}.
$$

The upper bound given in display (\ref{5.35}) for the Goto numbers
of monomial parameter ideals may fail to be sharp as we demonstrate
in Example \ref{5.4}.

\begin{example} \label{5.4}
Let $(R,\m)$ be a numerical semigroup ring associated to the
semigroup $G = \langle 4,7,9 \rangle$.  Then  $\rho = \max\{ g(x^4),
g(x^7), g(x^9) \} = 2$. However,  the Frobenius number $f = 10$ and
$$
\max\left\lbrace \left\lceil\frac{f + a_1 + 1 }{a_2}\right\rceil -
1, \left\lfloor\frac{a_2 - b_2 + f }{ a_1}\right\rfloor,
\left\lfloor\frac{a_3 - b_3 + f }{ a_1}\right\rfloor \right\rbrace =
3.
$$
\end{example}

Theorem \ref{5.5} shows  that the inequalities in Propositions
\ref{5.1} and \ref{5.3} are equalities if $d = 2$. We use the
well-known fact that if $G = \langle a_1, a_2 \rangle$, then the
Frobenius number $f = a_1a_2 - a_1 - a_2$, cf.
\cite[Example~12.2.1]{SHbk}.

\begin{theorem} \label{5.5}
Let $(R,\m)$ be a numerical semigroup ring associated to the
semigroup  ~$G = \langle a_1, a_2 \rangle$, then
\begin{eqnarray*}
g(x^{a_1}) &= &\left\lceil\frac{f + a_1 + 1 }{a_2}\right\rceil - 1 = a_1 - 1 \\
             &\le & g(x^{a_2}) = \left\lfloor\frac{a_2 - b_2 + f }{
a_1}\right\rfloor = a_2 - 1 - \left\lfloor\frac{a_2 - 1 }{
a_1}\right\rfloor.
\end{eqnarray*}
\end{theorem}

\begin{proof}
Using that $f = a_1a_2 - a_1 - a_2$, we see that $\left\lceil\frac{f
+ a_1 + 1 }{a_2}\right\rceil - 1 = a_1 - 1$.  Thus
Proposition~\ref{5.3} implies that $g(x^{a_1}) \le a_1 - 1$. Since
$a_1a_2 - a_1 - a_2 \notin G$, we have $(x^{a_2})^{a_1-1} \notin
x^{a_1}R$. Therefore $(x^{a_1}R:\m^{a_1-1}) \subseteq \m$, and thus
is integral over $x^{a_1}R$. Hence the Goto number $g(x^{a_1}) = a_1
- 1$.

Let $s =  \left\lfloor\frac{a_2 - b_2 + f }{ a_1}\right\rfloor$.  It
is clear that $b_2 = \left\lfloor\frac{a_2 - 1 }{ a_1}\right\rfloor
a_1$. Thus  $s = a_2 - 1 - \left\lfloor\frac{a_2 - 1 }{
a_1}\right\rfloor$. Proposition~\ref{5.1} implies that $g(x^{a_2})
\le s$. If $g(x^{a_2}) < s$, then for some nonnegative integer
$e < \frac{a_2}{a_1}$, we have $ea_1 + sa_1 - a_2 \in G$. Since
$s = a_2 - 1 - \left\lfloor\frac{a_2 - 1 }{
a_1}\right\rfloor$, we have
$$
ea_1 + sa_1 - a_2 = (e - \left\lfloor\frac{a_2 - 1 }{
a_1}\right\rfloor)a_1 + f.
$$
But $( \left\lfloor\frac{a_2 - 1 }{a_1}\right\rfloor -e)a_1 \in G$
implies that $f \in G$, a contradiction. Hence $g(x^{a_2}) = s$.

It remains to prove that $g(x^{a_1}) \le g(x^{a_2})$. Let $r_i \in
[0, a_i-1] \cap \mathbb N$, $1 \le i \le 2$, be such that
$\left\lceil\frac{f + a_1 + 1 }{ a_2}\right\rceil = \frac{f + a_1 +
1 + r_2 }{a_2}$,  and such that $\left\lfloor\frac{a_2 - 1 }{
a_1}\right\rfloor = \frac{a_2 - 1 - r_1 }{a_1}$. Then $g(x^{a_1})
\le g(x^{a_2})$ if and only if $\frac{f + a_1 + 1 + r_2 - a_2 }{
a_2} \le \frac{(a_2-1)a_1 - (a_2 - 1 - r_1) }{ a_1}$, which holds if
and only if $r_2 a_1 - r_1 a_2 \le a_1 a_2^2 + a_1 a_2 - a_2^2 + a_2
- a_1^2 a_2 - a_1$. But $r_2 a_1 - r_1 a_2 \le r_2 a_1 \le (a_2 - 1)
a_1$, and it suffices to prove that $(a_2 - 1) a_1 \le a_1 a_2^2 +
a_1 a_2 - a_2^2 + a_2 - a_1^2 a_2 - a_1$. By assumption $a_1 + 1 \le
a_2$, so that $a_2(a_1^2 - a) \le a_2^2 (a_1 - 1)$, which expands to
the desired inequality. This completes  the proof of
Theorem~\ref{5.5}.
\end{proof}

Now we turn to further characterizations of the eventual stable Goto
number of parameter ideals.

\begin{proposition} \label{5.6}
Let $(R,\m)$ be a numerical semigroup ring associated to the
semigroup $G = \langle a_1,\cdots, a_d \rangle$, and let $f$ denote
the Frobenius number of $R$. Let $t$ be the maximum integer such
that for all $\alpha \in \{1, 2, \ldots, a_1\}$, $\m^t \not
\subseteq x^{\alpha} R$ ($R$-module containment). Then $t =
g(x^{f+a_1+1})$.
\end{proposition}

\begin{proof}
Since $d > 1$, we have $\m \not \subseteq x^{\alpha} R$ for all
prescribed $\alpha$. Thus there exist positive integers $k$ such
that $\m^k \not \subseteq x^{\alpha} R$. There is an upper bound on
such $k$, for if $k$ is such that  $(k-2)a_1 > f$, then  $\m^k
\subseteq x^{\alpha} R$. Thus an integer  $t$ as in the statement of
Proposition~\ref{5.6}  exists.

If $x^l \in (x^{f + a_1 + 1}) : \m^t$, then  by possibly multiplying
by a power of $x^{a_1}$, we may assume without loss of generality
that $l \ge f + 1$. Then $\m^t \subseteq (x^{a_1+f+1-l})R$, so that
by the definition of $t$, $l \ge f+a_1+1$. This proves that $t \le
g(x^{f+a_1+1})$.

Also by the definition of $t$, there exists $\alpha \in \{1, 2,
\ldots, a_1\}$ such that $\m^{t+1} \subseteq x^{\alpha}R$. Then
$x^{f + a_1 + 1 - \alpha} \cdot \m^{t+1} \subseteq (x^{f + a_1 +
1})R$. Hence  $t + 1 > g(x^{f + a_1 + 1})$.
\end{proof}

\begin{proposition} \label{5.7}
Let $(R,\m)$ be a numerical semigroup ring associated to the
semigroup $G = \langle a_1,\cdots, a_d \rangle$, and let $f$ denote
the Frobenius number of $R$. For each $\alpha \in \{1, \ldots,
a_1\}$, find elements $\beta \in G$ such that $\beta - \alpha \not
\in G$. Among all such $\beta$, fix one for which $x^{\beta}$ has
the largest $\m$-adic order. As $\alpha$ varies, let $t'$ be the
smallest of these orders. Then $t' = g(x^{f+a_1+1})$.
\end{proposition}

\begin{proof}
Observe that the $\beta$ as in the statement exist: $0$ works, or by
minimality of the generators, either $a_1$ or $a_2$ work for each
$\alpha$. Necessarily $\beta - \alpha \le f$. For each $\alpha$, let
$\beta_{\alpha}$ be an element in $G$ such that $\beta_{\alpha} -
\alpha \not \in G$ and such that $x^{\beta_{\alpha}}$ has the
largest $\m$-adic order. Let $t$ be as in Proposition~\ref{5.6}.
Note that $t = g(x^{f+a_1+1})$.

Let $\alpha$ be such that the corresponding $\beta_{\alpha}$ yields
the smallest order, namely $t'$. By assumption, $\m^t \not \subseteq
x^{\alpha} R$. Thus there exists $\gamma \in G$ such that
$x^{\gamma} \in \m^t$ and $\gamma - \alpha \not \in G$. Hence $t'
\ge t$.

Now suppose that $t' > t$. Then there exists $\alpha \in \{1,
\ldots, a_1\}$ such that $\m^{t'} \subseteq x^{\alpha} R$. Hence the
$\m$-adic order of $x^{\beta_{\alpha}}$ is strictly smaller than
$t'$, which is a contradiction. Thus $t' \le t$.
\end{proof}

\begin{corollary} \label{5.8}
With notation as in Proposition~\ref{5.7}, the $\m$-adic order of
the conductor ideal $C = x^{f+1}V$ is less than or equal to  the Goto
number $g(x^{f+a_1+1})$.
\end{corollary}

\begin{proof}
For each $\alpha \in \{1, \ldots, a_1\}$, the element $f + \alpha$ is in $G$ and has the
property that subtracting $\alpha$ gives an element not in $G$. Hence
Proposition~\ref{5.7} implies Corollary~\ref{5.8}.
\end{proof}

With notation as in Corollary~\ref{5.8}, the element $f +
\alpha$ is the largest element of $G$ having the property that
subtracting $\alpha$ gives an element not in $G$. However,
in general,  $x^{f + \alpha}$ need not have the largest possible
$\m$-adic order,  as we demonstrate in Example~\ref{5.9}.

\begin{example} \label{5.9}
Let $(R,\m)$ be a numerical semigroup ring associated to the
semigroup $G = \langle 7, 9, 20 \rangle$.
The Frobenius number  $f = 33$,
and the $\m$-adic order of $x^{33 + 7} = x^{40} = x^{f+a_1}$ is $2$,
whereas $38 \in G$, $38-7 \not \in G$,
and the $\m$-adic order of $x^{38}$ is $3$.
In fact, the Goto number $g(x^{f+a_1+1}) = 3$,
so the $\m$-adic order of the conductor ideal $C = x^{f+1}V$
is here strictly smaller than $g(x^{f+a_1+1})$.
\end{example}

\begin{theorem} \label{5.10}
Let $(R,\m)$ be a numerical semigroup ring associated to the
semigroup $G = \langle a_1, a_2 \rangle$, then
$$
g(x^{f+a_1+1}) = g(x^{a_1}) = \min\{g(x^{a_i}) : i = 1, 2\}.
$$
\end{theorem}

\begin{proof}
The last equality is proved in Theorem~\ref{5.5}. There it was also
proved that $g(x^{a_1}) = a_1 - 1$. By Theorem~\ref{3.1},
$g(x^{f+a_1+1}) \le g(x^{a_1})$. By Proposition~\ref{5.6}, to prove
Theorem~\ref{5.10},  it suffices to prove for all $\alpha \in \{1,
\ldots, a_1\}$, that $\m^{a_1 - 1} \not \subseteq x^{\alpha} R$. Let
$r \in \{0, \ldots, a_1 - 1\}$ be such that $r a_2 \equiv -\alpha
\mod a_1$. Such $r$ exists because $a_1$ and $a_2$ are relatively
prime. Then $(x^{a_1})^r (x^{a_2})^{a_1-1-r} \in \m^{a_1 - 1}$.
Observe that $a_1a_2 - a_1 - a_2 - (r a_1 + (a_1-1-r) a_2 - \alpha)
= - a_1 + r (a_2 - a_1) + \alpha$ is by construction of $r$ an
integer multiple of $a_1$. If it were negative, then $- a_1 + r (a_2
- a_1) + \alpha \le -a_1$, which is a contradiction. So $a_1a_2 -
a_1 - a_2 - (r a_1 + (a_1-1-r) a_2 - \alpha)$ is a non-negative
multiple of $a_1$. Thus that non-negative multiple of $a_1$ plus $r
a_1 + (a_1-1-r) a_2 - \alpha$ equals $a_1a_2 - a_1 - a_2 = f$, which
is not in $G$. Hence $r a_1 + (a_1-1-r) a_2 - \alpha$ is not in $G$,
which proves that $\m^{a_1 - 1} \not \subseteq x^{\alpha} R$.
\end{proof}

\end{document}